\documentclass[reqno]{amsart}

\usepackage{amssymb,amsmath,amsthm,latexsym,booktabs,amscd,extarrows,tikz,mathrsfs,tikz-cd}
\usepackage[colorlinks=true,citecolor=red,linkcolor=blue]{hyperref}

\numberwithin{equation}{section}

\theoremstyle{definition}
\newtheorem{definition}{Definition}[section]
\newtheorem{remark}[definition]{Remark}

\theoremstyle{plain}
\newtheorem{lemma}[definition]{Lemma}

\newtheorem{corollary}[definition]{Corollary}

\newtheorem{proposition}[definition]{Proposition}

\begin{document}

\title{Embedding of pseudotensor category}

\author{Rui Yao; Zhixiang Wu}

\address{School of Mathematical Sciences, Zhejiang University, Hangzhou, P.R.China}

\email{yaorui@zju.edu.cn, wzx@zju.edu.cn}

\begin{abstract}
We realize the embedding functor from pseudotensor category to tensor category in a purely algebraic setting when the pseudotensor category is the category $\mathcal{M}(H)$ of left $H$-modules,
which is originally defined by Beilinson and Drinfeld. Then we use operadic methods to construct the Schur functor and free object in the tensor category.
\end{abstract}

\maketitle


\tableofcontents

\section{Introduction}
The notion of vertex algebra is a rigorous mathematical definition of the chiral part of a 2-dimensional quantum field theory, which was formulated by R.Borcherds in \cite{Borcherds}.
Conformal algebras provide an axiomatic description of the "singular" part of the vertex algebras and regard the vertex algebra as a high-weight module \cite{Kac}.
Using the language of pseudotensor categories, a natural generalization of conformal algebras was introduced in \cite{BDK}, which is called pseudoalgebras.
In \cite{BDK}, a full classification of semisimple finite Lie pseudoalgebras was obtained.
Classification of all finite irreducible representations of simple finite Lie $H$-pseudoalgebras has been completed in \cite{BDK1,BDK2,BDK3}.
Unital associative, Leibniz, graded left-symmetric, Jordan, Zinbiel pseudoalgebras were studied in \cite{RA,Wu1,Wu2,Kolesnikov,YR}.

An algebra of a certain type is usually defined by generating operations and relations. The notion of operad is an algebraic device which encodes a type of algebras.
Instead of studying the properties of a particular algebra, we focus on the universal operations that can be performed on the elements of any algebra of a given type.
The information contained in an operad consists in these operations and all the ways of composing them.
The operadic point of view has several advantages, such as applying many results known for classical types of algebras to other types of algebras and simplifying both the statements and the proofs.
The operad theory also plays an important role in the development of deformation theory and quantum field theory.
The original definition is due to Peter May, and was given in \cite{M}, where an operad is regarded as a collection of operations equipped with a notion of composition and subject to various conditions.
In the encyclopedic book \cite{LV} written by J.L.Loday and B.Vallette, it provides a comprehensive account of algebraic operad theory, containing introduction to algebraic operads,
a conceptual treatment of Koszul duality and applications to homotopical algebra.

The use of operadic viewpoint to study pseudoalgebra quickly emerged after the concept of pseudoalgebra was proposed. We can construct the following one-to-one correspondence
\[
    \mathscr{P}\ H\text{-pseudoalgebra structure on}\ M\leftrightarrow \text{Hom}_\text{Op}(\mathscr{P},\text{End}_M)
\]
where $H$ is a Hopf algebra, $M$ is a $H$-module, $\mathscr{P}$ is a algebraic operad and
$\text{End}_M$ has some modifications defined by $\text{End}_M(n)=\text{Hom}_{H^{\otimes n}}(M^{\boxtimes n},H^{\otimes n}\otimes_H M)$ for $n\geq0$.
However, the further research on the operadic morphism from $\mathscr{P}$ to $\text{End}_M$ seems to be difficult.
We have no specific methods to explain clearly how this operadic morphism is mapped.
The main obstruction is that we are not yet able to construct the free $\mathscr{P}\ H$-pseudoalgebra structure on a fixed $H$-module.
Hence, the various kinds of tools in the operad theory can not be applied directly to the study of pseudoalgebra.

When considering chiral algebra(vertex algebra) and $\text{Lie}^*$ algebra(Lie conformal algebra), there are many works arising by using operadic methods.
The pioneering work is dual to A. Beilinson and V. Drinfeld in their book \cite{BD}, where they give chiral algebra and chiral homology in a taste of algebraic geometry
via pseudotensor category and $\mathcal{D}$-module language. The chiral Koszul duality between chiral algebra and factorization algebras are studied in \cite{FG}.
In \cite{Ba,Ba2}, the authors translate the construction of the chiral opeard to the purely algebraic language of vertex algebras
and give rise to vertex algebra cohomology. The article \cite{NY} study the algebraic operad of SUSY vertex algebra.
There is also research on quantum field theory by colored operad; see \cite{BSW} and the references therein.

In the book \cite{BD}, it is proved that any pseudotensor category $\mathcal{M}$ can be realized as a full pseudotensor subcategory
of a tensor(=symmetric monoidal) category $\mathcal{M}^\otimes$. Then they give a specific construction of $\mathcal{M}^\otimes$, denoted by $\mathcal{M}(X^\mathcal{S})$,
when pseudotensor category $\mathcal{M}=\mathcal{M}(X)$ is the category of right $\mathcal{D}$-modules on a smooth curve $X$.

In this paper, we realize the above embedding functor and the tensor category in a purely algebraic setting when the pseudotensor category is the category $\mathcal{M}(H)$ of left $H$-modules.
We construct two tensor categories, which are called the category of $H^\infty$-modules and the category of interconnected $H^\infty$-modules,
denoted by $H^\infty\text{-Mod}$ and $H^\infty\text{-Mod}_{\text{IC}}$.
There are two paths embedding the pseudotensor category $\mathcal{M}(H)$ to the tensor categories.
The first one is the functor $\iota:\mathcal{M}(H)\hookrightarrow H^\infty\text{-Mod}\hookrightarrow H^\infty\text{-Mod}_{\text{IC}},V\mapsto V^\infty\mapsto(V^\infty,\delta)$
and the second one is the fully faithful functor $i:\mathcal{M}(H)\hookrightarrow H^\infty\text{-Mod}_{\text{IC}},V\mapsto(V^\infty,\text{id})$.
We show that our $H^\infty\text{-Mod}_{\text{IC}}$ is actually a generalization of $\mathcal{M}(X^\mathcal{S})$; see details in remark \ref{rem3}.

After the construction of the tensor category, the ground has been prepared for the algebraic operad.
In these two tensor categories, we can naturally consider many construction by operadic methods, such as (co)homology, (co)bar construction and Koszul duality.

In this setting, we are also hopeful to think about chiral algebra and factorization algebra structure over a higher dimensional scheme
and give chiral homology structure in a purely algebraic language, which may be easier for computation. We leave these for future work.

This paper is organized as follows. In Section 2, we recall some Hopf algebra theory, definition of pseudoalgebra and some operad theory.
In Section 3, we give the definition of $H^\infty$-modules and $H^\infty$-algebras, and construct two tensor products on the category of $H^\infty$-modules.
We also compare the relationship between $H$-pseudoalgebras and $H^\infty$-algebras.
In Section 4, another tensor category consisting of interconnected $H^\infty$-modules is defined and
we give rise to two functors embedding the pseudotensor category into the tensor category.
We also compare our construction with the tensor category $\mathcal{M}(X^\mathcal{S})$ given in \S 3.4.10 of \cite{BD}.
In Section 5, we use operadic methods to give the Schur functor and free object in the tensor category $H^\infty\text{-Mod}_\text{IC}$.

\section{Preliminaries}
\subsection{Hopf algebras}
In this subsection, we introduce some Hopf algebra theory that will be used later.
Let $H$ be a Hopf algebra with a coproduct $\Delta$, a counit $\varepsilon$, and an antipode $S$.
We will use the following notation:
\[\Delta (h)=h_{(1)}\otimes h_{(2)},\]
\[(\Delta\otimes id)\Delta(h)=(id\otimes\Delta)\Delta(h)=h_{(1)}\otimes h_{(2)}\otimes h_{(3)},\]
\[(S\otimes id)\Delta(h)=h_{(-1)}\otimes h_{(2)}.\]
The axioms of the antipode and the counit can be written as follows:
\[h_{(-1)}h_{(2)}=h_{(1)}h_{(-2)}=\varepsilon (h),\]
\[\varepsilon (h_{(1)})h_{(2)}=h_{(1)}\varepsilon (h_{(2)})=h.\]
There is a natural structure of $H$-module on $H^{\otimes n}$ given by
\begin{equation}\label{fo1}
(f_1\otimes\dots\otimes f_n)h=(f_1h_{(1)})\otimes\dots\otimes (f_nh_{(n)}),\quad f_i,h\in H.
\end{equation}
For any arbitrary Hopf algebra $H$, there is a map $\mathcal{F}:H\otimes H\rightarrow H\otimes H$, called the \textbf{Fourier transform}, by the formula
\[
    \mathcal{F}(f\otimes g)=(f\otimes 1)(S\otimes id)\Delta (g)=fg_{(-1)}\otimes g_{(2)}.
\]
It follows from the formula $h_{(-1)}h_{(2)}\otimes h_{(3)}=1\otimes h=h_{(1)}h_{(-2)}\otimes h_{(3)}$ that $\mathcal{F}$ is a vector space isomorphism with an inverse given by
\[
    \mathcal{F}^{-1}(f\otimes g)=(f\otimes 1)\Delta(g)=fg_{(1)}\otimes g_{(2)}.
\]
Using Fourier transform $\mathcal{F}$, for any $f\otimes g \in H\otimes H$, we have
\[f\otimes g =(fg_{(-1)}\otimes1)\Delta(g_{(2)}).\]

In this paper, we mainly consider cocommutative Hopf algebras, i.e., $h_{(1)}\otimes h_{(2)}=h_{(2)}\otimes h_{(1)}$ for all $h\in H$.
Suppose $H$ is a cocommutative Hopf algebra, and let $X=H^*$ be its dual algebra. We can fix a linear basis $\{h_i|i\in I\}$ of $H$ and denote its dual basis by $\{x_i|i\in I\}$,
i.e., $\langle x_i,h_j\rangle =\delta _{ij},i,j\in I$. There is an action of $H$ on $X$ given by
\[\langle xf,g\rangle =\langle x,gS(f)\rangle ,\langle fx,g\rangle =\langle x,S(f)g\rangle ,\]
for $x\in X, f,g\in H$. Moreover, there is a linear transformation $S:X\rightarrow X$ defined by $\langle S(x),f\rangle =\langle x,S(f)\rangle $ for $x\in X,f \in H$.

\subsection{\texorpdfstring{$H$-pseudoalgebras}{H-pseudoalgebras}}
Let $H$ be a cocommutative Hopf algebra over a field ${\bf k}$. Suppose that $A$ is a left $H$-module.
A pseudoproduct is an $H$-bilinear map $*:A\otimes A\rightarrow (H\otimes H)\otimes_H A$. An $H$-module $A$ endowed with a pseudoproduct $*$ is called an $H$-pseudoalgebra.
For every $n,m\geq 1$, an $H$-bilinear map $*$ can be naturally expanded to a map from $(H^{\otimes n}\otimes_H A)\otimes (H^{\otimes m}\otimes_H A)$ to $H^{\otimes (m+n)}\otimes_H A$:
\[
    (F\otimes_H a)*(G\otimes_H b)=((F\otimes G)\otimes_H 1)((\Delta^{n-1}\otimes \Delta^{m-1})\otimes_H id_Z)(a*b),
\]
where $F\in H^{\otimes n},G\in H^{\otimes m},a,b\in Z,\Delta^{n}=(\Delta\otimes id^{n-1})(\Delta\otimes id^{n-2})\dots (\Delta\otimes id)\Delta$.

\subsection{Operad}
By definition an $\mathbb{S}$-module over ${\bf k}$ is a family
\[
    M=(M(0),M(1),\dots,M(n),\dots),
\]
of right ${\bf k}[\mathbb{S}_n]$-modules $M(n)$. For $\mu\in M(n)$ the integer $n$ is called the arity of $\mu$.
A morphism of $\mathbb{S}$-modules $f:M\rightarrow N$ is a family of $\mathbb{S}_n$-equivariant maps $f_n:M(n)\rightarrow N(n)$.
To any $\mathbb{S}$-module $M$ we associate its Schur functor $\widetilde{M}:Vect\rightarrow Vect$ defined by 
\[
    \widetilde{M}(V):=\bigoplus_{n\geq 0}M(n)\otimes_{\mathbb{S}_n}V^{\otimes n}.
\]
From the direct sum, the tensor product and the composition of Schur functors, we can get the direct sum, the tensor product and the composition of $\mathbb{S}$-modules.
\[
    (M\oplus N)(n):=M(n)\oplus N(n),
\]
\[
    (M\otimes N)(n):=\bigoplus_{i+j=n}\text{Ind}^{\mathbb{S}_n}_{\mathbb{S}_i\times \mathbb{S}_j}M(i)\otimes N(j)\simeq \bigoplus_{i+j=n} M(i)\otimes N(j)\otimes {\bf k}[Sh(i,j)],
\]
\[
    (M\circ N)(n):=\bigoplus_{k\geq0} M(k)\otimes_{\mathbb{S}_k}(\bigoplus_{i_1+\dots+i_k=n} \text{Ind}^{\mathbb{S}_n}_{\mathbb{S}_{i_1}\times\dots\times \mathbb{S}_{i_k}}N(i_1)\otimes\dots\otimes N(i_k)),
\]
where Ind means induced representation, $Sh(i,j)$ means $(i,j)$-shuffles of $\mathbb{S}_n$.
\begin{definition}
    (\cite{LV}, \S 5.2.1)
    A \textbf{symmetric operad} $\mathscr{P}=(\mathscr{P},\gamma,\eta)$ is an $\mathbb{S}$-module $\mathscr{P}=\{\mathscr{P}(n)\}_{n\geq0}$
    endowed with morphisms of $\mathbb{S}$-modules
    \[
        \gamma:\mathscr{P}\circ \mathscr{P}\rightarrow \mathscr{P}
    \]
    called \textbf{composition map}, and
    \[
        \eta:I\rightarrow \mathscr{P}
    \]
    called \textbf{unit map}, which make $\mathscr{P}$ into a monoid.
    Explicitly, the morphisms $\gamma$ and $\eta$ satisfy associativity axiom:
    \begin{center}
        \begin{tikzpicture}
            \node(A) at (0,1.5) {$\mathscr{P}\circ(\mathscr{P}\circ\mathscr{P})$};
            \node(B) at (-3,0) {$(\mathscr{P}\circ\mathscr{P})\circ\mathscr{P}$};
            \node(C) at (-3,-2) {$\mathscr{P}\circ\mathscr{P}$};
            \node(D) at (4,-2) {$\mathscr{P}$};
            \node(E) at (4,1.5) {$\mathscr{P}\circ\mathscr{P}$};
            \draw[->] (B) -- node[above]  {$\cong$} (A);
            \draw[->] (B) -- node[left]  {$\gamma\circ Id$} (C);
            \draw[->] (C) -- node[above]  {$\gamma$} (D);
            \draw[->] (A) -- node[above]  {$Id\circ\gamma$} (E);
            \draw[->] (E) -- node[right]  {$\gamma$} (D);
        \end{tikzpicture}
    \end{center}
    and unitality axiom:
    \begin{center}
        \begin{tikzpicture}
            \node(A) at (0,0) {$I\circ\mathscr{P}$};
            \node(B) at (3,0) {$\mathscr{P}\circ\mathscr{P}$};
            \node(C) at (3,-1.5) {$\mathscr{P}$};
            \node(D) at (6,0) {$\mathscr{P}\circ I$};
            \draw[->] (A) -- node[above]  {$\eta\circ Id$} (B);
            \draw[->] (B) -- node[right]  {$\gamma$} (C);
            \draw[->] (A) -- node[below]  {=} (C);
            \draw[->] (D) -- node[above]  {$Id\circ\eta$} (B);
            \draw[->] (D) -- node[below]  {$=$} (C);
        \end{tikzpicture}
    \end{center}
\end{definition}
\begin{definition}
    (\cite{LV}, \S 5.2.3)
    An \textbf{algebra over the operad} $\mathscr{P}$, or a $\mathscr{P}$-\textbf{algebra} for short, is a vector space $A$ equipped with a linear map $\gamma_A:\mathscr{P}(A)\rightarrow A$
    such that the following diagrams commute:
    \begin{center}
    \begin{tikzpicture}
        \node(A) at (0,1.5) {$\mathscr{P}(\mathscr{P}(A))$};
        \node(B) at (-2,0) {$(\mathscr{P}\circ\mathscr{P})(A)$};
        \node(C) at (-2,-1.5) {$\mathscr{P}(A)$};
        \node(D) at (3,-1.5) {$A$};
        \node(E) at (3,1.5) {$\mathscr{P}(A)$};
        \draw[->] (B) -- node[above]  {=} (A);
        \draw[->] (B) -- node[left]  {$\gamma(A)$} (C);
        \draw[->] (C) -- node[above]  {$\gamma_A$} (D);
        \draw[->] (A) -- node[above]  {$\mathscr{P}(\gamma_A)$} (E);
        \draw[->] (E) -- node[right]  {$\gamma_A$} (D);
    \end{tikzpicture}
    \quad
    \begin{tikzpicture}
        \node(A) at (0,0) {I(A)};
        \node(B) at (2,0) {$\mathscr{P}(A)$};
        \node(C) at (2,-1.5) {A};
        \draw[->] (A) -- node[above]  {$\eta(A)$} (B);
        \draw[->] (B) -- node[right]  {$\gamma_A$} (C);
        \draw[->] (A) -- node[below]  {=} (C);
    \end{tikzpicture}
\end{center}
\end{definition}
There is an important kind of operad called \textbf{endomorphism operad}. For any vector space $V$ the endomorphism operad $\text{End}_V$ is given by
\[
    \text{End}_V(n):=\text{Hom}(V^{\otimes n},V),
\]
where, by convention, $V^{\otimes 0}={\bf k}$. The right action of $\mathbb{S}_n$ on $\text{End}_V$ is induced by the left action on $V^{\otimes n}$.
The composition map $\gamma$ is given by composition of endomorphisms:
\begin{center}
    \begin{tikzpicture}
        \node(A) at (0,0) {$V^{\otimes i_1}\otimes$};
        \node(B) at (2,0) {$\dots$};
        \node(C) at (4,0) {$\otimes V^{\otimes i_k}$};
        \node(D) at (6,0) {$=$};
        \node(E) at (8,0) {$V^{\otimes n}$};
        \node(A1) at (0,-1.5) {$V\otimes$};
        \node(B1) at (2,-1.5) {$\dots$};
        \node(C1) at (4,-1.5) {$\otimes V$};
        \node(D1) at (6,-1.5) {$=$};
        \node(E1) at (8,-1.5) {$V^{\otimes k}$};
        \node(B2) at (2,-3) {$V$};
        \node(D2) at (6,-3) {$=$};
        \node(E2) at (8,-3) {$V$};
        \draw[->] (A) -- node[right]  {$f_1$} (A1);
        \draw[->] (C) -- node[right]  {$f_k$} (C1);
        \draw[->] (E) -- node[right]  {$f_1\otimes\dots\otimes f_k$} (E1);
        \draw[->] (B1) -- node[right]  {$f$} (B2);
        \draw[->] (E1) -- node[right]  {$f$} (E2);
    \end{tikzpicture}
\end{center}
\[
    \gamma(f;f_1,\dots,f_k):=f(f_1\otimes\dots\otimes f_k).
\]
It is immediate to verify that $\text{End}_V$ is an algebraic operad.
\begin{proposition}
    (\cite{LV}, Prop 5.2.2)
    A $\mathscr{P}$-algebra structure on the vector space $A$ is equivalent to a morphism of operads $\mathscr{P}\rightarrow End_A$.
\end{proposition}
\begin{remark}\label{rem2}
    The definition of symmetric operad makes sense in any symmetric monoidal category $(\mathcal{V},\otimes,u)$.
    If $\mathcal{C}$ is a symmetric monoidal category enriched in $\mathcal{V}$, and we can define $\mathscr{P}$-algebras in $\mathcal{C}$
    for any operad $\mathscr{P}$ in $\mathcal{V}$, consisting of an object $X\in\mathcal{C}$ and a morphism of operads in $\mathcal{V}$
    \[
        \chi:\mathscr{P}\rightarrow \text{End}^{\mathcal{V}}_X.
    \]
\end{remark}

\section{\texorpdfstring{Basic theory of $H^\infty$-algebras}{Basic theory of H infty algebras}}
In this section, we will give the definition of $H^{\infty}$-algebras and their corresponding basic theory.

\subsection{\texorpdfstring{The structure of $H^\infty$}{The structure of H infty}}
Firstly, it is well-known that the category of Hopf algebras over a fixed field ${\bf k}$ has the structure of tensor(=symmetric monoidal) category with tensor product $\otimes:=\otimes_{\bf k}$ and unit object ${\bf k}$.
Let $H$ be any arbitrary Hopf algebra over ${\bf k}$. We can define $H^{\infty}=\oplus_{n\geq 0}H^{\otimes n}$, where $H^{\otimes0}={\bf k}$.
If $h\in H^{\otimes n}\subset H^\infty$, then the superscript $n$ is called \textbf{tensor degree} and we have the notation $td(x)=n$. Define by $\widehat{H^\infty}=\prod_{n\geq 0}H^{\otimes n}$ the completion of $H^\infty$.
Next, we will give a algebraic structure on $H^{\infty}$. For any $a_1\otimes\dots\otimes a_n\in H^{\otimes n}$ and $b_1\otimes\dots\otimes b_m\in H^{\otimes m}$,
define multiplication as follows:
\begin{align*}
    (a_1\otimes\dots\otimes a_n)(b_1\otimes\dots\otimes b_m)=
    \left\{
        \begin{array}{cc}
            a_1b_1\otimes\dots\otimes a_nb_n,&\text{if} \quad n=m;\\
            0,&\text{if} \quad n\neq m.
        \end{array}
    \right.
\end{align*}
Then $H^\infty$ ia an associative algebra. It is easy to see that $H^{\infty}$ is a commutative algebra if and only if $H$ is a commutative algebra.
When consider $\widehat{H^\infty}$, there is a unit denoted by $e:=(1,1_H,1_H^{\otimes2},\dots,1_H^{\otimes n},\dots)$, where $1$ is the unit of ${\bf k}$ and $1_H$ is the unit of $H$.

We can also define a comultiplication as follows:
\[
    \Delta^\infty(h_1\otimes\dots\otimes h_n)=1\otimes(h_1\otimes\dots\otimes h_n)+\sum_{i=1}^{n-1}(h_1\otimes\dots\otimes h_i)\otimes(h_{i+1}\otimes\dots\otimes h_n)+(h_1\otimes\dots\otimes h_n)\otimes 1,
\]
where $h_1,\dots,h_n\in H$.
It is easy to see that the comultiplication is coassociative.
Hence, $(H^\infty,\Delta^\infty,\varepsilon)$ has a coalgebra structure, where counit $\varepsilon:H^\infty\rightarrow k$ is defined by $u\mapsto1,h_1\otimes\dots\otimes h_n\mapsto0$.
Unfortunately, $H^\infty$ cannot have a Hopf algebra structure.

\subsection{\texorpdfstring{The structure of $H^\infty$-modules}{The structure of H infty modules}}
We will denote the category of left $H^\infty$-modules by $H^\infty$-Mod.
\subsubsection{Natural gradation}
For a left $H^{\infty}$-module $W$, let $W^n=(1_H^{\otimes n})W$.
We can see that there is a natural graded structure on $H^{\infty}$-module $W$, that is $W=\oplus_{n\geq 0} W^n$.
Hence, we can also define \textbf{tensor degree} on any $H^{\infty}$-module.
Since $W^n$ can be regarded as an $H^{\otimes n}$-module, a left $H^{\infty}$-module $W$ is actually the direct sum of $H^{\otimes n}$-module for all $n\geq0$.
There is also the completion of $W$ denoted by $\widehat{W}=\prod_{n\geq 0}W^n$.
\subsubsection{Tensor products}
Let $V,W$ be two $H^{\infty}$-modules. There is a natural $H^{\infty}$-module structure on $V\otimes W$, which is given by
\begin{flalign*}
    &(h_1\otimes\dots\otimes h_n)(v^p \otimes w^q)=&
\end{flalign*}
\begin{flalign*}
    &&
    \left\{
        \begin{array}{cc}
            (h_1\otimes\dots\otimes h_p)v^p \otimes (h_{p+1}\otimes\dots\otimes h_n)w^q,&\text{if} \quad n=p+q;\\
            0,&\text{if} \quad n\neq p+q,
        \end{array}
    \right.
\end{flalign*}
for any $h_1,\dots,h_n \in H$ and $v^p\in V^p, w^q\in W^q$. Hence, we define the tensor product of two $H^{\infty}$-modules.
We denote the tensor product of two $H^{\infty}$-modules $V,W$ with the above $H^{\infty}$-module structure by $V\boxtimes W$.
Actually, we can regard $V\boxtimes W$ as $\bigoplus_{n\geq0}(\bigoplus_{p+q=n}V^p\boxtimes W^q)$.
\subsubsection{Homomorphisms}
Next, we will study the homomorphisms of two $H^{\infty}$-modules.
\begin{lemma}\label{lem1}
    Let $V,W$ be two $H^{\infty}$-modules. For any $f\in \text{Hom}_{H^\infty}(V,W)$, we have $f(V^n)\subseteq W^n$.
\end{lemma}
\begin{proof}
    Suppose there is an element $v\in V^n$ such that $f(v)\in W^m,m\neq n$ and $f(v)\neq 0$. Then using $1_H^{\otimes m}$ to act on $f(v)$,
    we have $f(v)=1_H^{\otimes m}f(v)=f(1_H^{\otimes m}v)=f(0)=0$. So we get a contradiction.
\end{proof}
\begin{lemma}
    For any $H^\infty$-module $V$, we have the following isomorphisms of $H^\infty$-modules
    \[
        H^\infty\otimes_{H^\infty}V\cong V\quad\text{and}\quad\text{Hom}_{H^\infty}(H^\infty,V)\cong \widehat{V}.
    \]
\end{lemma}
\begin{proof}
    Define $f:V\rightarrow H^\infty\otimes_{H^\infty}V$, $v^n\mapsto1_H^{\otimes n}\otimes_{H^\infty}v^n$, where $v^n\in V^n$.
    We can see that $f$ is injective and surjective. For $h^\infty\in H^\infty$ with $td(h^\infty)=n$,
    we have $f(h^\infty v^n)=1_H^{\otimes n}\otimes_{H^\infty}h^\infty v^n=h^\infty\otimes_{H^\infty}v^n=h^\infty f(v^n)$.
    Hence, $H^\infty\otimes_{H^\infty}V\cong V$.

    The $H^\infty$-generators of $H^\infty$ are $1,1_H,1_H^{\otimes2},\dots,1_H^{\otimes n},\dots$.
    Define $g:\text{Hom}_{H^\infty}(H^\infty,V^\infty)\rightarrow\widehat{V^\infty}$, $\alpha\mapsto(\alpha(1),\alpha(1_H),\alpha(1_H^{\otimes2}),\dots,\alpha(1_H^{\otimes n}),\dots)$,
    where $\alpha(1_H^{\otimes n})\in V^n$. We can see that $g$ is also an injective and surjective homomorphism. 
    Hence, $\text{Hom}_{H^\infty}(H^\infty,V^\infty)\cong \widehat{V^\infty}$.
\end{proof}
\begin{remark}\label{rem1}
    Actually, we have $H^\infty\otimes_{H^\infty}V=\bigoplus_{n\geq0}H^{\otimes n}\otimes_{H^\infty}V^n\cong\bigoplus_{n\geq0}V^n=V^\infty$.
    Similarly, we have
    $\text{Hom}_{H^\infty}(H^\infty,V)=\text{Hom}_{H^\infty}(\bigoplus_{n\geq0}H^{\otimes n},\bigoplus_{n\geq0}V^n)\cong\prod_{n\geq0}\text{Hom}_{H^{\otimes n}}(H^{\otimes n},V^n)\cong\prod_{n\geq0}V^n=\widehat{V}$.
\end{remark}
\begin{proposition}\label{prop1}
    Let $V,W$ be two $H^{\infty}$-modules, we have the following decomposition
    \[
    \text{Hom}_{H^\infty}(V,W)=\prod_{n\geq0}\text{Hom}_{H^{\otimes n}}(V^n,W^n).
    \]
\end{proposition}
\begin{proof}
    It follows immediately from lemma \ref{lem1} and a similar discussion in remark \ref{rem1}.
\end{proof}

There is also $(H^\infty)^{\otimes n}$-module $V_1\otimes\dots\otimes V_n$ defined by the natural tensor products of n $H^\infty$-modules $V_1,\dots,V_n$.
We have $V_1\otimes\dots\otimes V_n=\bigoplus_{k_1,\dots,k_n\geq0}V_1^{k_1}\otimes\dots\otimes V_n^{k_n}$ and we can define tensor degree using the grading.
We can define an $(H^\infty)^{\otimes n}$-module structure on an $H^\infty$-module $V$
by the following canonical surjective map from $(H^\infty)^{\otimes n}$ to $H^\infty$
\begin{flalign*}
    &(h_1\otimes\dots\otimes h_{k_1})\otimes\dots\otimes(h_{k_1+\dots+k_{n-1}+1}\otimes\dots\otimes h_{k_1+\dots+k_n})\mapsto&
\end{flalign*}
\begin{flalign*}
    &&h_1\otimes\dots\otimes h_{k_1}\otimes\dots\otimes h_{k_1+\dots+k_{n-1}+1}\otimes\dots\otimes h_{k_1+\dots+k_n},
\end{flalign*}
for $h_1,\dots,h_{k_n}\in H,k_1,\dots,k_n\geq0$. When some $k_i=0$, the formal notation $h_{k_1+\dots+k_{i-1}+1}\otimes\dots\otimes h_{k_1+\dots+k_i}=1$.
\begin{lemma}\label{lem2}
    For any $H^\infty$-module $V$, we have the following isomorphism of sets
    \[
    \text{Hom}_{(H^\infty)^{\otimes n}}(V^{\otimes n},V)\cong\text{Hom}_{\bf k}((V^0)^{\otimes n},V^0).
    \]
\end{lemma}
\begin{proof}
    Let $f\in \text{Hom}_{(H^\infty)^{\otimes n}}(V^{\otimes n},V)$ and $v^1\otimes\dots\otimes v^n\in(H^\infty)^{\otimes n}$
    with $td(v^1\otimes\dots\otimes v^n)=(k_1,\dots,k_n)\neq(0,\dots,0)$. Suppose that $k_n\neq0$.
    Then we have
    \begin{flalign*}
        &f(v^1\otimes\dots\otimes v^n)=f(((1_H)^{\otimes k_1}\otimes\dots\otimes (1_H)^{\otimes k_n})(v^1\otimes\dots\otimes v^n))&
    \end{flalign*}
    \begin{flalign*}
        &&=(1_H)^{\otimes (k_1+\dots+k_n)}f(v^1\otimes\dots\otimes v^n)=f(((1_H)^{\otimes (k_1+\dots+k_n)}\otimes u\otimes\dots\otimes u)(v^1\otimes\dots\otimes v^n))=0.
    \end{flalign*}   
    So $f(v^1\otimes\dots\otimes v^n)\neq0$ only if $td(v^1\otimes\dots\otimes v^n)=(0,\dots,0)$.
    Hence, we can regard $f$ as an element which comes from $\text{Hom}_{{\bf k}^{\otimes n}}((V^0)^{\otimes n},V^0)\cong\text{Hom}_{\bf k}((V^0)^{\otimes n},V^0)$.
\end{proof}
\begin{remark}
    By lemma \ref{lem2}, it is not very interesting to study $(H^\infty)^{\otimes n}$-homomorphisms from $V^{\otimes n}$ to $V$.
    Alternatively, we will concentrate on $H^\infty$-homomorphisms from $V^{\boxtimes n}$ to $V$.
\end{remark}
\begin{proposition}
    For any $H^\infty$-module $V$, we have the following isomorphism of sets
    \[
    \text{Hom}_{H^\infty}(V^{\boxtimes n},V)\cong\prod_{\substack{k\geq0\\k_1+\dots+k_n=k}}\text{Hom}_{H^{\otimes k}}(\boxtimes_{i=1}^nV^{k_i},V^k).
    \]
\end{proposition}

\subsection{\texorpdfstring{Two tensor products on $H^\infty$-Mod}{Two tensor products on H infty Mod}}
The category of $H^\infty$-modules is a $\text{Vec}_{\bf k}$-enriched category. It is not difficult to verify that $H^\infty$-Mod with tensor product $\boxtimes$ is a monoidal category.
We have the following proposition.
\begin{proposition}
    The category of $H^\infty$-modules $(H^\infty\text{-Mod},\boxtimes,{\bf k})$ with $H^\infty$-homomorphisms is a monoidal category, where ${\bf k}$ is a $H^\infty$-module concentrated on tensor degree $0$,
    i.e., $(h_1\otimes\dots\otimes h_n)\cdot 1=0$ for any $h_1\otimes\dots\otimes h_n\in H^{\otimes n},n\geq1$ and $k\cdot 1=k$ for $k\in H^{\otimes0}={\bf k}$.
\end{proposition}
\begin{remark}
    The above monoidal category is not symmetric.
    For any $H^\infty$-modules $V,W$, we can define the following map
    \[
        c_{V,W}:V\boxtimes W\rightarrow W\boxtimes V,\quad v^n\boxtimes w^m\mapsto w^m\boxtimes v^n,
    \]
    where $v^n\in V^n,w^m\in W^m$. We can see that $c_{V,W}$ is not an $H^\infty$-homomorphism, but only a ${\bf k}$-map.
\end{remark}

In order to construct operad theory, we need a symmetric monoidal category of $H^\infty$-modules by remark \ref{rem2}.
Now, we will define a new symmetric tensor product $\otimes^*$ on $H^\infty$-Mod.
Let $V_i,i\in I$ be a finite non-empty family of $H^\infty$-modules. One has
\[
(\otimes^*_I V_i)^{|J|}:=\underset{J\rightarrow I}{\oplus}\underset{I}{\boxtimes}(V_i)^{|J_i|},
\]
where $J$ is a finite set and $H^{\otimes |J|}$-module structure on $\underset{I}{\boxtimes}(V_i)^{|J_i|}$ is given by composing with a twist, i.e.,
\[
H^{\otimes |J|}\xrightarrow{\sigma}H^{\otimes |J_1|}\otimes\dots\otimes H^{\otimes |J_{|I|}|}\rightarrow\text{End}(\underset{I}{\boxtimes}(V_i)^{|J_i|}).
\]
For $I=\emptyset$, we set $\otimes^*_I V_i={\bf k}$ additionally.
The explicit formula of $V\otimes^* W$ for $|J|\leq3$ will be given as follows.
\[
(V\otimes^* W)^0=V^0\boxtimes W^0,\quad (V\otimes^* W)^1=V^1\boxtimes W^0\oplus V^0\boxtimes W^1,
\]
\[
(V\otimes^* W)^2=V^2\boxtimes W^0\oplus V^1\boxtimes W^1\oplus W^1\boxtimes V^1\oplus V^0\boxtimes W^2,
\]
\[
(V\otimes^* W)^3=V^3\boxtimes W^0\oplus V^2\boxtimes W^1\oplus V^{|\{1,3\}|}\boxtimes W^{|\{2\}|}\oplus W^1\boxtimes V^2
\]
\[
\oplus V^1\boxtimes W^2\oplus V^{|\{2\}|}\boxtimes W^{|\{1,3\}|}\oplus W^2\boxtimes V^1\oplus V^0\boxtimes W^3.
\]
The new tensor products are associative and commutative in the obvious way, hence $(H^\infty\text{-Mod},\otimes^*,{\bf k})$ is a tensor category.

\subsection{\texorpdfstring{Definition of $H^\infty$-algebras}{Definition of H infty algebras}}
\begin{definition}
    Suppose that $V$ is an $H^\infty$-module. An {\bf $H^\infty$-algebra} is an $H^\infty$-module endowed with a product
    $*\in \text{Hom}_{H^\infty}(V\otimes^* V,V)$.
\end{definition}
By the definition of symmetric tensor product $\otimes^*$ and the natural $\mathbb{S}_2$-action on $\text{Hom}_{H^\infty}(V\otimes^* V,V)$,
we have the following isomorphism of sets
\[
\text{Hom}_{H^\infty}(V\otimes^* V,V)\cong\prod_{\substack{J=J_1\sqcup J_2\\\text{with }1\in J_1}}\text{Hom}_{H^{\otimes |J|}}(V^{|J_1|}\boxtimes V^{|J_2|},V^{|J|}).
\]
Let $W$ be another $H^\infty$-algebra. A \textbf{morphism of $H^\infty$-algebras} is an $H^\infty$-map $f:V\rightarrow W$ which commutes with $H^\infty$-products,
i.e., $f(v_1*_Vv_2)=f(v_1)*_Wf(v_2)$ for any $v_1,v_2\in V$. We denote by \textbf{$H^\infty$-alg} the category of $H^\infty$-algebras.

\subsection{\texorpdfstring{$H^{\infty}$-algebras and $H$-pseudoalgebras}{H infty algebras and H pseudoalgebras}}
In this subsection, we will explore the relationship between $H^\infty$-algebras and $H$-pseudoalgebras.

An $H$-pseudoalgebra structure over a $H$-module $V$ is a pseudoproduct $*\in \text{Hom}_{H^{\otimes2}}(V\boxtimes V, H^{\otimes2}\otimes_H V)$.
By expanding the pseudoproduct naturally, we can immediately get the following embedding map.
\begin{lemma}
    For the homomorphism set $\text{Hom}_{H^{\otimes2}}(V\boxtimes V, H^{\otimes2}\otimes_H V)$ of pseudoproducts, there is a natural embedding map for $m,n\geq1$
    \begin{flalign*}
        &\text{Hom}_{H^{\otimes2}}(V\boxtimes V, H^{\otimes2}\otimes_H V)\simeq \text{Hom}_{H^{\otimes2}}((H\otimes_H V)\boxtimes (H\otimes_H V), H^{\otimes2}\otimes_H V)&
    \end{flalign*}
\begin{flalign*}
    &&\hookrightarrow \text{Hom}_{H^{\otimes(m+n)}}((H^{\otimes m}\otimes_H V)\boxtimes(H^{\otimes n}\otimes_H V),H^{\otimes(m+n)}\otimes_H V).
\end{flalign*}
Moverover, for a pseudoproduct $*$, we can uniquely correspond to a pseudoproduct sequence
\begin{flalign*}
&(*_{(0,0)},*_{(0,1)},*_{(1,0)},*_{(0,2)},*_{(1,1)},*_{(2,0)},*_{(0,3)}*_{(1,2)},*_{(2,1)},*_{(3,0)},\dots)\in&
\end{flalign*}
\begin{flalign*}
&&\prod_{m,n\geq 0} \text{Hom}_{H^{\otimes(m+n)}}((H^{\otimes m}\otimes_H V)\boxtimes(H^{\otimes n}\otimes_H V),H^{\otimes(m+n)}\otimes_H V),
\end{flalign*}
where 
\begin{flalign*}
    &*_{(m,n)}:(h_1\otimes\dots\otimes h_m\otimes_H v)\boxtimes(t_1\otimes\dots\otimes t_n\otimes_H w)\mapsto&
\end{flalign*}
\begin{flalign*}
    &&(h_1\otimes\dots\otimes h_m\otimes t_1\otimes\dots\otimes t_n\otimes_H id)(\Delta^{m-1}\otimes\Delta^{n-1}\otimes_H id)(v*w),
\end{flalign*}
for any $h_1,\dots,h_m,t_1,\dots,t_n\in H$ and $v,w\in V$, where $\Delta^{-1}=\varepsilon$.
\end{lemma}
\begin{proposition}
    There is a set isomorphism 
    \begin{flalign*}
    &\prod_{m,n\geq 0} \text{Hom}_{H^{\otimes(m+n)}}((H^{\otimes m}\otimes_H V)\boxtimes(H^{\otimes n}\otimes_H V),H^{\otimes(m+n)}\otimes_H V)&
    \end{flalign*}
    \begin{flalign*}
    &&\simeq \text{Hom}_{H^\infty}(V^\infty\otimes^* V^\infty, V^\infty),
    \end{flalign*}
    where $V^\infty$ is a regular left $H^\infty$-module from a left $H$-module $V$.
\end{proposition}
\begin{proof}
    Fix a finite set $J$ and for any $J=J_1\sqcup J_2$ with $|J_1|=m,|J_2|=n$, an element in $\text{Hom}_{H^{\otimes |J|}}(V^{|J_1|}\boxtimes V^{|J_2|},V^{|J|})$
    is uniquely determined by an element in $\text{Hom}_{H^{\otimes(m+n)}}((H^{\otimes m}\otimes_H V)\boxtimes(H^{\otimes n}\otimes_H V),H^{\otimes(m+n)}\otimes_H V)$
    by the cocommutativity of $H$ and the natural $\mathbb{S}_2$-action.
    Hence, for any sequence
    \[(\mu_{(0,0)},\mu_{(0,1)},\mu_{(1,0)},\mu_{(0,2)},\mu_{(1,1)},\mu_{(2,0)},\mu_{(0,3)}\mu_{(1,2)},\mu_{(2,1)},\mu_{(3,0)},\dots),\]
    there is a uniquely determined sequence $(\mu_{(J_1.J_2)})_{J=J_1\sqcup J_2}$.
    Then we can have a unique $\mu_\infty$ defined as follows
    \[
        \mu_\infty|_{\text{Hom}_{H^{\otimes |J|}}(V^{|J_1|}\boxtimes V^{|J_2|},V^{|J|})}=\mu_{(J_1,J_2)}.
    \]
    On the other hand, for any $\mu_\infty\in \text{Hom}_{H^\infty}(V^\infty\otimes^* V^\infty, V^\infty)$, we can also have a unique sequence
    \[(\mu_\infty| _{(0,0)},\mu_\infty| _{(0,1)},\mu_\infty| _{(1,0)},\mu_\infty| _{(0,2)},\mu_\infty| _{(1,1)},\mu_\infty| _{(2,0)},\mu_\infty| _{(0,3)},\mu_\infty| _{(1,2)},\dots).\]
\end{proof}
Now we can construct an $H^\infty$-algebra structure on regular $V^\infty$. Define $H^\infty$-multiplication $*_\infty$ as follows
\begin{flalign*}
    &(h_1\otimes\dots\otimes h_m\otimes_H v)*_\infty (t_1\otimes\dots\otimes t_n\otimes_H w)&
\end{flalign*}
\begin{flalign*}
    =(h_1\otimes\dots\otimes h_m\otimes_H v)*_{(m,n)}(t_1\otimes\dots\otimes t_n\otimes_H w)
\end{flalign*}
\begin{flalign*}
    &&=(h_1\otimes\dots\otimes h_m\otimes t_1\otimes\dots\otimes t_n\otimes_H id)(\Delta^{m-1}\otimes\Delta^{n-1}\otimes_H id)(v*w),
\end{flalign*}
for any $h_1,\dots,h_m,t_1,\dots,t_n\in H$ and $v,w\in V$, where $\Delta^{-1}=\varepsilon$.

Denote by $\text{Hom}^{pseudo}_{H^\infty}(V^\infty\otimes^* V^\infty, V^\infty)$ the set of $H^\infty$-multiplications derived from pseudoproducts,
which is obviously a subset of $\text{Hom}_{H^\infty}(V^\infty\otimes^* V^\infty, V^\infty)$.
Then we can get the following corollary.
\begin{corollary}
    There is an isomorphism
    \[
        Hom_{H^{\otimes2}}(V\boxtimes V, H^{\otimes2}\otimes_H V)\simeq \text{Hom}^{pseudo}_{H^\infty}(V^\infty\otimes^* V^\infty, V^\infty).
    \]
\end{corollary}
As a result, we know that $H^\infty$-algebra structure is a larger set containing $H$-pseudoalgebra structure as its subset.

\section{\texorpdfstring{Category setting of $H^{\infty}$-algebras}{Category setting of H infty algebras}}\label{sec1}
In this section, we will study categories and functors related to $H^\infty$-algebras. Let $H$ be a cocommutative Hopf algebra over a field ${\bf k}$.
\subsection{\texorpdfstring{From $H$-modules to $H^{\infty}$-modules}{From H-modules to H infty-modules}}
Let $V$ be a left $H$-module. We can naturally define a left $H^{\infty}$-module $V^{\infty}=H^{\infty}\otimes_H V$.
The $H^\infty$-module structure on $V^{\infty}$ is given by
\begin{align*}
    (h_1\otimes\dots\otimes h_n)(a_1\otimes\dots\otimes a_m\otimes_H v)=
    \left\{
        \begin{array}{cc}
            h_1a_1\otimes\dots\otimes h_na_n\otimes_H v,&\text{if} \quad n=m;\\
            0,&\text{if} \quad n\neq m,
        \end{array}
    \right.
\end{align*}
for any $h_1,\dots,h_n,a_1,\dots,a_n\in H$ and $v\in V$. There is an isomorphism
\[
H^{\infty}\otimes_H V=(\oplus_{n\geq 0}H^{\otimes n})\otimes_H V\simeq \oplus_{n\geq 0}(H^{\otimes n}\otimes_H V),
\]
which gives rise to the gradation of $V^\infty$ from the gradation of $H^\infty$, i.e., $(V^\infty)^n=H^{\otimes n}\otimes_H V$ for $n\geq0$.
We call $V^{\infty}$ a \textbf{regular} left $H^{\infty}$-module from a left $H$-module $V$.

Denote by $\mathcal{M}(H)$ the category of left $H$-modules. The above construction actually gives an embedding functor $\iota$ from $\mathcal{M}(H)$ to $H^\infty$-Mod.
For any $H$-map $f:V\rightarrow W$, we have the following $H^\infty$-map
\[
\iota(f):V^\infty\rightarrow W^\infty,\quad h_1\otimes\dots\otimes h_n\otimes_Hv\mapsto h_1\otimes\dots\otimes h_n\otimes_Hf(v),
\]
where $n\geq0$ and $h_1,\dots,h_n\in H,v\in V$. Furthermore, we can extend the functor $\iota$ as follows.

Denote by $\hat{\mathcal{S}}$ the category of arbitrary finite sets and maps between them. Let $\pi:I\rightarrow J$ be a map in $\hat{\mathcal{S}}$.
Then the map can be factorized as $J\twoheadrightarrow\pi(J)\hookrightarrow I$.
For a surjective map $\pi:J\rightarrow I$, we have a decomposition $J=J_1\sqcup\dots\sqcup J_{|I|}$.
There is a right $H$-action on every $H^{\otimes|J_i|}$ by formula \ref{fo1} and then we have a right $H^{\otimes|I|}$-action on $H^{\otimes|J|}$ by tensor them together.
For an injective map $\pi:J\rightarrow I$, there is a right $H^{\otimes|I|}$-module structure on $H^{\otimes|J|}$ by
regarding it as $H^{\otimes|\pi(J)|}\boxtimes{\bf k}^{\boxtimes|I\backslash\pi(J)|}$.

Denote by $\mathcal{M}(H^{\otimes n})$ the category of left $H^{\otimes n}$-modules.
Fix $n\geq1$, for any $H^{\otimes n}$-map $f^n:V^n\rightarrow W^n$, let
\[
(V^n)^\infty=\underset{\pi\in\hat{\mathcal{S}}:J\rightarrow\{1,\dots,n\}}{\bigoplus}H^{\otimes|J|}\otimes_{H^{\otimes n}}V^n
\]
and we can give the following $H^\infty$-map
\[
\iota(f^n):(V^n)^\infty\rightarrow(W^n)^\infty,\quad h_1\otimes\dots\otimes h_m\otimes_{H^{\otimes n}}v\mapsto h_1\otimes\dots\otimes h_m\otimes_{H^{\otimes n}}f^n(v),
\]
where $m\geq0$ and $h_1,\dots,h_m\in H,v\in V^n$.
In other words, this gives functor $\mathcal{M}(H^{\otimes n})\hookrightarrow H^\infty\text{-Mod}$.
When $n=0$, let $\mathcal{M}(H^{\otimes0})={\bf k}\text{-Mod}\hookrightarrow H^\infty\text{-Mod}$ be the trivial embedding functor.
Therefore, we have an extending functor of $\iota$ from $\mathcal{M}(H^{\otimes\leq n}):=\amalg_{0\leq k\leq n}\mathcal{M}(H^{\otimes k})$ to $H^\infty$-Mod.

Since $\lim_{n\rightarrow\infty}\mathcal{M}(H^{\otimes\leq n})=\amalg_{k\geq0}\mathcal{M}(H^{\otimes k})=H^\infty\text{-Mod}$, we have an endofunctor over $H^\infty$-Mod, also denoted by $\iota$.
We will discribe the functor more explicitly. Let $V=\oplus_{n\geq0}V^n$ be an $H^\infty$-module.
We have $\iota(V)=\oplus_{n\geq0}\iota(V^n)=V^0\oplus(\oplus_{n\geq1}(V^n)^\infty)$. Then
\[
(\iota(V))^0=\bigoplus_{m\geq0}{\bf k}\otimes_{H^{\otimes m}}V^m,
\]
and for $n\geq1$,
\[
(\iota(V))^n=\underset{\pi\in\hat{\mathcal{S}}:\{1,\dots,n\}\rightarrow I}{\bigoplus}H^{\otimes n}\otimes_{H^{\otimes|I|}}V^{|I|}
\]
Let $f:V\rightarrow W$ be an $H^\infty$-homomorphism. By proposition \ref{prop1}, we have $f=(f^0,f^1,f^2,\dots)$, where $f^n:V^n\rightarrow W^n$ is an $H^{\otimes n}$-map.
Then we have $(\iota(f))^n=f^n+\sum_{m\neq n}\iota(f^m)$ for $n\geq0$.

\subsection{\texorpdfstring{Interconnected $H^\infty$-modules}{Interconnected H infty modules}}
Next, we will construct a bigger category consisting of $H^\infty$-modules with extra interconnection.
Let $V=\oplus_{n\geq0}V^n$ be an $H^\infty$-module. Let $\theta_V$ be a rule that assigns to any map $\pi:J\rightarrow I$ in $\hat{\mathcal{S}}$
an $H^{\otimes |J|}$-homomorphism $\theta^{(\pi)}=\theta^{(\pi)}_V:H^{\otimes|J|}\otimes_{H^{\otimes|I|}}V^{|I|}\rightarrow V^{|J|}$.
We demand that the $\theta^{(\pi)}$ are compatible with the composition of the $\pi$'s, i.e.,
$\theta^{(\pi_1\pi_2)}=\theta^{(\pi_1)}(H^{\otimes|K|}\otimes_{H^{\otimes|J|}}\theta^{(\pi_2)})$ for $\pi_1:K\rightarrow J,\pi_2:J\rightarrow I$
and $\theta^{(\text{id}_I)}=\text{id}_{V^{|I|}}$. We call $(V,\theta_V)$ an \textbf{interconnected} $H^\infty$-module.

Let $(W,\theta_W)$ be another interconnected $H^\infty$-module.
A \textbf{morphism of interconnected $H^\infty$-modules} is an $H^\infty$-map $f:V\rightarrow W$ which commutes with interconnection rules, i.e.,
$f\circ\theta^{(\pi)}_V=\theta^{(\pi)}_W\circ(H^{\otimes|J|}\otimes_{H^{\otimes|I|}}f)$ for any map $\pi:J\rightarrow I$.
Denote by $H^\infty\text{-Mod}_\text{IC}$ the category of interconnected $H^\infty$-modules.

It is not difficult to see that the interconnection rule $\theta_V$ can be regarded as an $H^\infty$-map from $\iota(V)$ to $V$ and
when restricted to $H^{\otimes|J|}\otimes_{H^{\otimes|\pi(J)|}}(H^{\otimes|\pi(J)|}\otimes_{H^{\otimes|I|}}V^{|I|})$ the map is equal to $\theta^{(\pi)}$.
Hence, the condition of an $H^\infty$-map $f:V\rightarrow W$ commuting with interconnection rules can be written as the diagram below.
\begin{center}
    \begin{tikzpicture}
        \node(A) at (0,0) {$\iota(V)$};
        \node(B) at (0,-1.5) {$V$};
        \node(C) at (3,0) {$\iota(W)$};
        \node(D) at (3,-1.5) {$W$};
        \draw[->] (A) -- node[left]  {$\theta_V$} (B);
        \draw[->] (A) -- node[above]  {$\iota(f)$} (C);
        \draw[->] (B) -- node[above]  {$f$} (D);
        \draw[->] (C) -- node[right]  {$\theta_W$} (D);
    \end{tikzpicture}
\end{center}

Furthermore, we have the following proposition for morphisms in $H^\infty\text{-Mod}_\text{IC}$.
\begin{proposition}\label{prop2}
    Let $V^1\neq0$ and $\theta_1$ be the interconnection map $\theta$ restricted to $\iota(V^1)$, which is an $H^\infty$-map from $(V^1)^\infty$ to $V$.
    Then a morphism of interconnected $H^\infty$-modules from $V$ to $W$ is equivalent to an $H^\infty$-map $f:V\rightarrow W$ satisfying $f\circ(\theta_V)_1=(\theta_W)_1\circ\iota(f^1)$.
\end{proposition}
\begin{proof}
    We only need to show that the remaining diagrams are commutative, i.e.,
    $f^{|J|}\circ\theta^{(\pi)}_V=\theta^{(\pi)}_W\circ(H^{\otimes|J|}\otimes_{H^{\otimes|I|}}f^{|I|})$ holds for any map $\pi:J\rightarrow I$ with $|I|>1$.
    Let $\pi_1:I\rightarrow\{1\},\pi_2:J\rightarrow\{1\}$, then $\pi\pi_1=\pi_2$ and we have the following diagram.
    \begin{center}
        \begin{tikzpicture}
            \node(A) at (0,0) {$H^{\otimes|I|}\otimes_HV^1$};
            \node(B) at (5,0) {$H^{\otimes|I|}\otimes_HW^1$};
            \node(C) at (0,-1.5) {$V^{|I|}$};
            \node(D) at (5,-1.5) {$W^{|I|}$};
            \node(E) at (0,-3) {$H^{\otimes|J|}\otimes_{H^{\otimes|I|}}V^{|I|}$};
            \node(F) at (5,-3) {$H^{\otimes|J|}\otimes_{H^{\otimes|I|}}W^{|I|}$};
            \node(G) at (0,-4.5) {$V^{|J|}$};
            \node(H) at (5,-4.5) {$W^{|J|}$};
            \draw[->] (A) -- node[above]  {$\iota(f^1)$} (B);
            \draw[->] (A) -- node[left]  {$\theta^{(\pi_1)}_V$} (C);
            \draw[->] (B) -- node[right]  {$\theta^{(\pi_1)}_W$} (D);
            \draw[->] (C) -- node[above]  {$f^{|I|}$} (D);
            \draw[->] (C) -- node[left]  {$\iota$} (E);
            \draw[->] (E) -- node[above]  {$\iota(f^{|I|})$} (F);
            \draw[->] (D) -- node[right]  {$\iota$} (F);
            \draw[->] (E) -- node[left]  {$\theta^{(\pi)}_V$} (G);
            \draw[->] (F) -- node[right]  {$\theta^{(\pi)}_W$} (H);
            \draw[->] (G) -- node[above]  {$f^{|J|}$} (H);
            \draw[->][bend left=60] (B.east) to node[right]{$\theta^{(\pi_2)}_W$} (H.east);
            \draw[->][bend right=60] (A.west) to node[left]{$\theta^{(\pi_2)}_V$} (G.west);
        \end{tikzpicture}
    \end{center}

    Since $\theta^{(\pi_1)}$ and $\theta^{(\pi_2)}$ are parts of $\theta_0$, we know that the first rectangle and the big circle are commutative.
    The second rectangle commutes by the functority of $\iota$. By the compatibility of $\theta$ with the composition, both the left and right semicircles are commutative.
    Therefore, the remaining third rectangle commutes, and we complete the proof.
\end{proof}
\begin{remark}
    If $V^n\neq0$, let $\theta_n$ be the interconnection map $\theta$ restricted to $\iota(V^n)$,
    we can replace the condition in proposition \ref{prop2} with $f\circ(\theta_V)_n=(\theta_W)_n\circ\iota(f^n)$ by a similar discussion.
\end{remark}
We can verify that $H^\infty\text{-Mod}_\text{IC}$ is an abelian category.
There is an exact fully faithful embedding $i:\mathcal{M}(H)\hookrightarrow H^\infty\text{-Mod}_\text{IC}$ defined by $i(V)=(V^\infty,\text{id})$ for any $H$-module $V$,
where $\text{id}^{(\pi)}:=\text{id}_{H^{\otimes|J|}\otimes_{H^{\otimes|I|}}V^{|I|}}$, for any map $\pi:J\rightarrow I$.
This embedding is left adjoint to the projection functor $p:H^\infty\text{-Mod}_\text{IC}\rightarrow\mathcal{M}(H),(V,\theta_V)\mapsto V^1$.
Furthermore, for any $n>1$ we can define an exact fully faithful embedding $\mathcal{M}(H^{\otimes n})\hookrightarrow H^\infty\text{-Mod}_\text{IC}$ by
$V^n\mapsto((V^n)^\infty,\text{id})$ for any $H^{\otimes n}$-module $V^n$.
Therefore, when $n$ is large enough, we have an exact fully faithful embedding from $H^\infty$-Mod to $H^\infty\text{-Mod}_\text{IC}$
sending $V$ to $(\iota(V),\text{id})$, also denoted by $i$.

For any $H^\infty$-module $V$, define a special interconnection rule $\delta$ by assigning to any permutation map $\sigma:I\rightarrow I$ an $H^{\otimes|I|}$-homomorphism
\begin{flalign*}
&h_1\otimes\dots\otimes h_{|I|}\otimes_{H^{\otimes|\sigma(I)|}}v=1_H^{\otimes|I|}\otimes_{H^{\otimes|\sigma(I)|}}(h_{\sigma^{-1}(1)}\otimes\dots\otimes h_{\sigma^{-1}(|I|)})v&
\end{flalign*}
\begin{flalign*}
&&\mapsto(h_{\sigma^{-1}(1)}\otimes\dots\otimes h_{\sigma^{-1}(|I|)})v,
\end{flalign*}
where $h_1,\dots,h_{|I|}\in H,v\in V^{|I|}$, and to other maps $0$-homomorphisms.
Then there is also a natural embedding from $H^\infty$-Mod to $H^\infty\text{-Mod}_\text{IC}$ defined by $V\mapsto(V,\delta_V)$, which is also denoted by $\iota$.
Note that $i$ and $\iota$ are different: $i$ is fully faithful, while $\iota$ is not.

The tensor products can be defined on $H^\infty\text{-Mod}_\text{IC}$ naturally. Let $\{(V_i,\theta_{V_i})\}_{i\in I}$ be an $I$-family of interconnected $H^\infty$-modules.
For $i\in I$ and $\pi_i\in\hat{\mathcal{S}}:K_i\rightarrow J_i$, we define $\otimes_I\pi_i:\sqcup_I K_i\rightarrow\sqcup_I J_i$ by $(\otimes_I\pi_i)(k)=\pi_i(k)$ if $k\in K_i$.
Then we have $\theta_{\boxtimes_IV_i}^{(\otimes_I\pi_i)}:=\theta_{V_1}^{(\pi_1)}\boxtimes\dots\boxtimes\theta_{V_{|I|}}^{(\pi_{|I|})}$ from
$\boxtimes_I(H^{\otimes|K_i|}\otimes_{H^{\otimes|J_i|}}V_i^{|J_i|})=H^{\otimes|\sqcup_IK_i|}\otimes_{H^{\otimes|\sqcup_IJ_i|}}(\boxtimes_IV_i^{|J_i|})$ to $\boxtimes_IV_i^{|K_i|}$.
Hence, the tensor product of interconnected $H^\infty$-modules is given by $\boxtimes_I(V_i,\theta_{V_i})=(\boxtimes_IV_i,\theta_{\boxtimes_IV_i})$.

By definition, there is $(\otimes^*_I V_i)^{|K|}=\underset{K\rightarrow I}{\oplus}\underset{I}{\boxtimes}(V_i)^{|K_i|}$.
Then we can define $\theta_{\otimes^*_IV_i}^{(\otimes_I\pi_i)}:=\bigoplus_{K\rightarrow I}\theta_{V_1}^{(\pi_1)}\boxtimes\dots\boxtimes\theta_{V_{|I|}}^{(\pi_{|I|})}$ from
$\bigoplus_{K\rightarrow I}\boxtimes_I(H^{\otimes|K_i|}\otimes_{H^{\otimes|J_i|}}V_i^{|J_i|})$ to $(\otimes^*_I V_i)^{|K|}$.
Hence, we also get the symmetric tensor product of interconnected $H^\infty$-modules, which is defined by $\otimes^*_I(V_i,\theta_{V_i})=(\otimes^*_IV_i,\theta_{\otimes^*_IV_i})$.
\begin{remark}
    We can see that locally $\otimes^*$ is the same as $\boxtimes$, i.e., $V^n\otimes^* W^m=V^n\boxtimes W^m$ for any $H^{\otimes n}$-module $V^n$ and $H^{\otimes m}$-module $W^m$.
\end{remark}
Now consider $\pi_i\in\hat{\mathcal{S}}:K_i\rightarrow J_i$ with $|J_i|=1$.
\begin{proposition}\label{prop3}
There is a beautiful equality given by
\[
V_1^\infty\otimes^*\dots\otimes^* V_{|I|}^\infty=(V_1\boxtimes\dots\boxtimes V_{|I|})^\infty.
\]
\end{proposition}
\begin{proof}
    The proof is completed directly when we see that the left hand side is equal to $\bigoplus_{K\rightarrow I}\boxtimes_I(H^{\otimes|K_i|}\otimes_H V_i)$
    and the right hand side is equal to $\bigoplus_{K\rightarrow I}H^{\otimes|K|}\otimes_{H^{\otimes|I|}}(\boxtimes_I V_i)$.
\end{proof}
Then we have the following proposition which is similar to proposition \ref{prop2}.
\begin{proposition}\label{prop4}
    Let $V_i^1\neq0$ and $(\theta_{V_i})_1$ be the interconnection map $\theta_{V_i}$ restricted to $\iota(V_i^1)$, which is an $H^\infty$-map from $(V_i^1)^\infty$ to $V_i$.
    Then the interconnection map $\theta_{\otimes^*_IV_i}$ restricted to $\iota(\boxtimes_IV_i)$ is equal to $\otimes^*_I(\theta_{V_i})_1$, denoted by $(\theta_{\otimes^*_IV_i})_n$.
    Hence a morphism of interconnected $H^\infty$-modules from $\otimes^*_IV_i$ to $W$ is equivalent to an $H^\infty$-map $f:\otimes^*V_i\rightarrow W$
    satisfying $f\circ(\theta_{\otimes^*_IV_i})_n=(\theta_W)_1\circ\iota(f^n)$.
\end{proposition}

\subsection{\texorpdfstring{Definition of interconnected $H^\infty$-algebras}{Definition of interconnected H infty algebras}}
\begin{definition}
    Suppose that $(V,\theta_V)$ is an interconnected $H^\infty$-module. An \textbf{interconnected $H^\infty$-algebra} is an interconnected $H^\infty$-module
    endowed with a product $*\in\text{Hom}_{H^\infty}(V\otimes^* V,V)$ commuting with interconnection rules.
\end{definition}
Let $(W,\theta_W)$ be another interconnected $H^\infty$-algebra. A \textbf{morphism of interconnected $H^\infty$-algebras} is an interconnected $H^\infty$-map
$f:(V,\theta_V)\rightarrow(W,\theta_W)$ which commutes with interconnected $H^\infty$-products, i.e., $f(v_1*_Vv_2)=f(v_1)*_Wf(v_2)$, for any $v_1,v_2\in V$.
We denote by \textbf{$H^\infty\text{-alg}_{\text{IC}}$} the category of interconnected $H^\infty$-algebras.

\subsection{\texorpdfstring{Interconnected $H^{\infty}$-algebras and $H$-pseudoalgebras}{Interconnected H infty algebras and H pseudoalgebras}}
The following proposition is a direct consequence of full faithness of the embedding functor $i:\mathcal{M}(H)\hookrightarrow H^\infty\text{-Mod}_{\text{IC}}$
and we want to explain it more explicitly here.
\begin{proposition}
    For any regular $H^\infty$-module $V^\infty$ from a $H$-module $V$, there is an isomorphism of vector spaces
    \[
    \text{Hom}_{H^\infty\text{-Mod}_{\text{IC}}}(V^\infty\otimes^*V^\infty,V^\infty)\cong\text{Hom}_{H^{\otimes2}}(V\boxtimes V,H^{\otimes2}\otimes_H V).
    \]
\end{proposition}
\begin{proof}
    By proposition \ref{prop3} and proposition \ref{prop4}, we have the following commutative diagram.
    \begin{center}
    \begin{tikzpicture}
        \node(A) at (0,0) {$V\boxtimes V$};
        \node(B) at (0,-1.5) {$H^{\otimes2}\otimes_HV$};
        \node(C) at (3,0) {$\substack{(V\boxtimes V)^\infty\\=V^\infty\otimes^*V^\infty}$};
        \node(D) at (3,-1.5) {$V^\infty$};
        \node(E) at (6,0) {$V^\infty\otimes^* V^\infty$};
        \node(F) at (6,-1.5) {$V^\infty$};
        \draw[->] (A) -- node[left]  {$*^2$} (B);
        \draw[->] (A) -- node[above]  {$\iota$} (C);
        \draw[->] (B) -- node[above]  {$\iota$} (D);
        \draw[->] (C) -- node[right]  {$\iota(*^2)$} (D);
        \draw[->] (C) -- node[above]  {id} (E);
        \draw[->] (D) -- node[above]  {id} (F);
        \draw[->] (E) -- node[right]  {$*$} (F);
    \end{tikzpicture}
\end{center}
By the functority of $\iota$, we have
\[
\iota(*^2)\in\iota(\text{Hom}_{H^{\otimes2}}(V\boxtimes V,H^{\otimes2}\otimes_H V))\cong\text{Hom}_{H^{\otimes2}}(V\boxtimes V,H^{\otimes2}\otimes_H V).
\]
Since the interconnection rules are identity map, we complete the proof.
\end{proof}

\subsection{Relationship with pseudotensor category}
Denote by $\mathcal{S}$ the category of finite non-empty sets and surjective maps.
For a morphism $\pi:J\twoheadrightarrow I$ in $\mathcal{S}$ and $i\in I$ set $J_i:=\pi^{-1}(i)\subset J$; let $\cdot\in\mathcal{S}$ be the one element set.
\begin{definition}\label{def1}
    A \textbf{pseudotensor category} is a class of objects $\mathcal{M}$ together with the following datum:\\
    (a) For any $I\in\mathcal{S}$, an $I$-family of objects $L_i\in\mathcal{M},i\in I$, and an object $M\in\mathcal{M}$,
    one has $I$-operations: $P_I^{\mathcal{M}}(\{L_i\},M)=P_I(\{L_i\},M)$. There is a natural symmetric group $\mathbb{S}_I$-action on $P_I$ by permuting the inputs.\\
    (b) For any surjective map $\pi:J\twoheadrightarrow I$, families of objects $\{L_i\}_{i\in I},\{K_j\}_{j\in J}$, and an object $M$,
    one has the composition map
    \begin{equation}\label{eq2}
        P_I(\{L_i\},M)\times\prod_IP_{J_i}(\{K_j\},L_i)\rightarrow P_J(\{K_j\},M),\quad(\varphi,(\psi_i))\mapsto\varphi(\psi_i).
    \end{equation}
    The following properties should hold for the compositions:\\
    \textbf{Associativity}: If $H\twoheadrightarrow J$ is another surjective map, $\{F_h\}$ an $H$-family of objects, $\chi_j\in P_{H_j}(\{F_h\},K_j)$,
    then $\varphi(\psi_i(\chi_j))=(\varphi(\psi_i))(\chi_j)\in P_H(\{F_h\},M)$.\\
    \textbf{Unit}: For any $M\in\mathcal{M}$ there is an element $\text{id}_M\in P_\cdot(\{M\},M)$ such that for any $\varphi\in P_I(\{L_i\},M)$
    one has $\text{id}_M(\varphi)=\varphi(\text{id}_{L_i})=\varphi$.\\
    \textbf{Equivariance}: The compositions \ref{eq2} are equivariance with respect to the natural action of the symmetric group.
\end{definition}
\begin{remark}
    The notion of pseudotensor category is a straightforward generalization of the notion of operad. By definition, an operad is a pseudotensor category with only one object.
    Pseudotensor categories were considered sporadically and under various names;
    see \cite{La} ("multicategories") or \cite{Li,B2} ("multilinear categories") or \cite{Y} (“colored operads”).
\end{remark}
If in definition \ref{def1} we replace $\mathcal{S}$ by $\hat{\mathcal{S}}$, then we get the definition of \textbf{augmented pseudotensor category}.
Let $\mathcal{M}$ be any pseudotensor category. Suppose we have a functor $h:\mathcal{M}\rightarrow\text{Sets}$,
and for any finite set $I$ with $|I|\geq2$ and $i_0\in I$ we have natural maps
\[
h_{I,i_0}:P_I(\{L_i\},M)\times h(L_{i_0})\rightarrow P_{I\backslash\{i_0\}}(\{L_i\},M),
\]
where $\{L_i\}$ is an $I$-family of objects and $M\in\mathcal{M}$. We call these data an \textbf{augmentation functor} if the following compatibilities hold:\\
(a) Let $J$ be a finite set with $|J|\geq2$, $\pi:J\rightarrow I$ a surjective map, $j_i\in J$.
Then for $\varphi\in P_I(\{L_i\},M),\psi_i\in P_{J_i}(\{K_j\},L_i),a\in h(K_{j_0})$ one has
\[
h_{J,j_0}(\varphi(\psi_i),a)=\varphi(\psi_{i^\prime},h_{J_{i_0},j_0}(\psi_{i_0},a))
\]
(here $i_0=\pi(j_0),i^\prime\in I\backslash\{i_0\}$) if $|J_{i_0}|\geq2$, and
\[
h_{J,j_0}(\varphi(\psi_i),a)=h_{I,i_0}(\varphi,\psi_{i_0}(a))(\psi_{i^\prime})
\]
if $|J_{i_0}|=1$.\\
(b) Assume that $|I|\geq2$; let $i_0,i_1\in I$ be two distinct elements. Then for $\varphi\in P_I(\{L_i\},M),a_{i_0}\in h(L_{i_0}),a_{i_1}\in h(L_{i_1})$ one has
\[
h_{I\backslash\{i_0\},i_1}(h_{I,i_0}(\varphi,a_0),a_1)=h_{I\backslash\{i_1\},i_0}(h_{I,i_1}(\varphi,a_1),a_0)\in P_{I\backslash\{i_0,i_1\}}(\{L_i\},M)
\]
if $|I|>2$, and
\[
h_{I,i_0}(\varphi,a_0)a_1=h_{I,i_1}(\varphi,a_1)a_0\in h(M)
\]
if $|I|=2$.\\
One may consider an augmented pseudotensor category as a usual pseudotensor category equipped with an augmentation functor.

In the book \cite{BD} written by A. Beilinson and V. Drinfeld, it is proved that any pseudotensor category $\mathcal{M}$ can be realized as a full pseudotensor subcategory
of a tensor category $\mathcal{M}^\otimes$. They give a universal construction of such an embedding $\mathcal{M}\hookrightarrow\mathcal{M}^\otimes$ as follows.
Define an object of $\mathcal{M}^\otimes$ by a collection $\{M_i\}_{i\in I}$ of objects of $\mathcal{M}$ labeled by some finite non-empty set $I$,
denoted by $\otimes M_i=\otimes_IM_i$. Define a morphism of $\varphi:\otimes_JN_j\rightarrow\otimes_IM_i$ by a collection $(\pi,\{\varphi_i\}_{i\in I})$
where $\pi:J\twoheadrightarrow I$ and $\varphi_i\in P_{J_i}(\{N_j\},M_i)$, denoted by $\varphi=\otimes\varphi_i=\otimes_\pi\varphi_i$.

Let $X$ be a smooth curve and $\mathcal{M}(X)$ be the category of right $\mathcal{D}$-modules on $X$
($:=$ sheaves of $\mathcal{D}_X$-modules quasi-coherent as $\mathcal{O}_X$-modules). The category $\mathcal{M}(X)$ has a pseudotensor structure by setting
$P_I^*(\{L_i\},M):=\text{Hom}(\boxtimes L_i,\Delta^{(I)}_*M)$. In this case, they embed $\mathcal{M}(X)$ in a larger abelian category $\mathcal{M}(X^{\mathcal{S}})$
so that the pseudotensor structure on $\mathcal{M}(X)$ are induced from the tensor structure on $\mathcal{M}(X^{\mathcal{S}})$,
which is a specific construction of $\mathcal{M}^\otimes$; see \S 3.4.10 of \cite{BD} for further details.
\begin{remark}\label{rem3}
    We explain here that our $H^\infty\text{-Mod}_{\text{IC}}$ is actually a generalization of $\mathcal{M}(X^\mathcal{S})$.
    Given a cocommutative Hopf algebra $H$ over a field ${\bf k}$ and let $A$ be a commutative associative $H$-differential algebra, i.e.,
    a left $H$-module with a commutative associative multiplication; for example, $A=H^*:=\text{Hom}_{\bf k}(H,{\bf k})$.
    Then the algebra of differential operators on the affine scheme $X:=\text{Spec}(A)$ is $\mathcal{D}(X)=A\sharp H$, where $\sharp$ means smash product.
    Since $H$ is a Hopf algebra, $G:=\text{Spec}(H)$ has an algebraic group structure. The $G$-action on $X$ is given by $H$-action on $A$.
    Then the category of $G$-equivariant $\mathcal{D}$-modules on $X$ is equivalent to the category of $H$-modules.
    In the above setting, when $H={\bf k}[\partial]$ the pseudotensor category $\mathcal{M}^G(X)$(here superscript $G$ means $G$-equivariant) is our pseudotensor category $\mathcal{M}(H)$,
    and the tensor category $\mathcal{M}^G(X^{\mathcal{S}})$ is a tensor subcategory of our tensor category $H^\infty\text{-Mod}_{\text{IC}}$ whose objects are
    $(V,\theta_V)$ with $V^0=0$ and $\theta^{(\pi)}_V=0$ for any $\pi\in\hat{\mathcal{S}}\backslash\mathcal{S}$.
    The fully faithful pseudotensor embedding functor $\Delta_*^{(\mathcal{S})}:\mathcal{M}^G(X)\hookrightarrow\mathcal{M}^G(X^{\mathcal{S}})$
    is our functor $i:\mathcal{M}(H)\hookrightarrow H^\infty\text{-Mod}_{\text{IC}}$.
\end{remark}

\section{Operadic methods in the tensor categories}
Let $\mathscr{P}$ be a symmetric operad over ${\bf k}$, $\mathcal{M}$ be a pseudotensor ${\bf k}$-category,
and $F:\mathscr{P}\rightarrow\mathcal{M}$ be a pseudotensor functor, where $\mathscr{P}$ can be regarded as a pseudotensor category with single object $\cdot$.
For any object $V\in\mathcal{M}$, we denote by $\text{End}_V:=P^\mathcal{M}(\{V\},V)$ the endomorphism operad over $V$.
Then there is a morphism of operad $\mathscr{P}\rightarrow\text{End}_{F(\cdot)}$ induced from $F$, also denoted by $F$.
We know that a morphism $F$ is equivalent to a $\mathscr{P}$-pseudoalgebra structure on $F(\cdot)$.

Since there are no (pseudo)tensor products on pseudotensor category when the pseudotensor structure is not representable,
we can not define a $\mathscr{P}$-pseudoalgebra structure by Schur functor.
Therefore, many useful construction in representation theory cannot be derived naturally by operadic methods; for example, free object and coalgebra structure.
As a result, in order to using operadic methods, we prefer a tensor category structure and we need to embed a pseudotensor category into a proper tensor category.

Let $H$ be a Hopf algebra over a field ${\bf k}$.
From the discussion in section \ref{sec1}, we have two paths embedding the pseudotensor category $\mathcal{M}(H)$ to a proper tensor category.
The first one is
\[
\iota:\mathcal{M}(H)\hookrightarrow H^\infty\text{-Mod}\hookrightarrow H^\infty\text{-Mod}_{\text{IC}},\quad V\mapsto V^\infty\mapsto(V^\infty,\delta),
\]
and the second one is
\[
i:\mathcal{M}(H)\hookrightarrow H^\infty\text{-Mod}_{\text{IC}},\quad V\mapsto(V^\infty,\text{id}).
\]

Let $V$ be an $H^\infty$-module. We can define a $\mathscr{P}\ H^\infty$-algebra(resp. $\mathscr{P}$ interconnected $H^\infty$-algebra) structure on $V$
as an operadic morphism $\mathscr{P}\rightarrow\text{End}_V$,
where $\text{End}_V(n):=\text{Hom}_{H^\infty}(V^{\otimes^* n},V)$(resp. $\text{Hom}_{H^\infty\text{-Mod}_{\text{IC}}}(V^{\otimes^* n},V)$).
Furthermore, we have the category of $\mathscr{P}\ H^\infty$-algebra(resp. $\mathscr{P}$ interconnected $H^\infty$-algebra),
denoted by $\mathscr{P}\ H^\infty\text{-alg}$(resp. $\mathscr{P}\ H^\infty\text{-alg}_{\text{IC}}$).

In these two categories, we can naturally consider many construction by operadic methods, such as (co)homology, (co)bar construction and Koszul duality.
We refer the readers to this book \cite{LV} written by J.L. Loday and B. Vallette for more details.

\subsection{Schur functor}
In this subsection, we will give the definition of Schur functor in the tensor category $H^\infty\text{-Mod}_{\text{IC}}$.

Let $M$ be any $\mathbb{S}$-module. The \textbf{Schur functor} in $H^\infty\text{-Mod}_{\text{IC}}$ associated to $M$ is defined by
\[
\widetilde{M}:H^\infty\text{-Mod}_\text{IC}\rightarrow H^\infty\text{-Mod}_{\text{IC}}\quad V\mapsto\bigoplus_{n\geq0}M(n)\otimes_{\mathbb{S}_n}V^{\otimes^*n}.
\]
Here $V^{\otimes^*n}$ is viewed as a left $\mathbb{S}_n$-module under the left action
\[
\sigma\cdot(v_1,\dots,v_n):=(v_{\sigma^{-1}(1)},\dots,v_{\sigma^{-1}(n)}),
\]
where $v_1,\dots,v_n\in V$.
Any morphism of $\mathbb{S}$-modules $\alpha:M\rightarrow N$ gives rise to a transformation of functors $\widetilde{\alpha}:\widetilde{M}\rightarrow\widetilde{N}$.
If the $\mathbb{S}$-module $M$ is concentrated in arity $0$(resp. $1$, resp. $n$),
then the functor $\widetilde{M}$ is constant(resp. linear, resp. homogeneous polynomial of degree $n$).
We can get the \textbf{identity functor}, denoted by $\widetilde{I}$, by taking the Schur functor of $I:=(0,{\bf k},0,0,\dots)$,
so $\widetilde{I}(V)=\text{Id}_{H^\infty\text{-Mod}_{\text{IC}}}(V)=V$.

There are three important constructions on endofunctors of $H^\infty\text{-Mod}_{\text{IC}}$: the direct sum, the tensor product and the composition,
which are given by
\[
(F\oplus G)(V):=F(V)\oplus G(V),
\]
\[
(F\otimes G)(V):=F(V)\otimes G(V),
\]
\[
(F\circ G)(V):=F(G(V)).
\]
It is not difficult to show that if the functors $F$ and $G$ are Schur functors, then the resulting functor is also a Schur functor.

\subsection{Free object}
We will give here the construction of free object in the category $\mathscr{P}\ H^\infty\text{-alg}_{\text{IC}}$.

Let $\mathscr{U}:\mathcal{C}\rightarrow\mathcal{D}$ be a functor such that the objects of $\mathcal{D}$ are the same as the objects of $\mathcal{C}$ except that
we do not take into account some of the data. Such a functor is called a \textbf{forgetful functor}.
We often denote by $C$ instead of $\mathscr{U}(C)$ for the image of an object $C$ by a forgetful functor.
Let $\mathscr{F}:\mathcal{D}\rightarrow\mathcal{C}$ be a functor left adjoint to $\mathscr{U}$, i.e.,
\[
\text{Hom}_\mathcal{C}(\mathscr{F}(D),C)\cong\text{Hom}_\mathcal{D}(D,\mathscr{U}(C)).
\]
The image of an object $D$ of $\mathcal{D}$ by $\mathscr{F}$ is called a \textbf{free object}.
The free object $\mathscr{F}(D)\in\mathcal{C}$ over the object $D\in\mathcal{D}$ is characterized by the following property:

For any object $C\in\mathcal{C}$ and any morphism $f:D\rightarrow\mathscr{U}(C)$ in $\mathcal{D}$,
there exists a unique morphism $\widetilde{f}:\mathscr{F}(D)\rightarrow C$ in $\mathcal{C}$, i.e., the following diagram commutative:
\begin{center}
    \begin{tikzpicture}
        \node(A) at (0,0) {$D$};
        \node(B) at (2.5,0) {$\mathscr{F}(D)$};
        \node(C) at (2.5,-1.5) {$C$};
        \draw[->] (A) -- node[above]  {$\eta$} (B);
        \draw[->] (B) -- node[right]  {$\widetilde{f}$} (C);
        \draw[->] (A) -- node[below]  {$f$} (C);
    \end{tikzpicture}
\end{center}
Observe that a free object $\mathscr{F}(D)$ is unique up to a unique isomorphism.

Now we consider the free $\mathscr{P}$ interconnected $H^\infty$-algebra. Let $(V,\theta_V)$ be an interconnected $H^\infty$-module.
Then $\mathscr{P}(V)$ is the image of $V$ by the Schur functor $\mathscr{P}$.
Let $\gamma:\mathscr{P}\circ\mathscr{P}\rightarrow\mathscr{P}$ and $\eta:I\rightarrow\mathscr{P}$ be the composition map and unit map in the definition of operad $\mathscr{P}$.
The defining axioms of the operad $\mathscr{P}$ show that $(\mathscr{P}(V),\gamma(V):\mathscr{P}(\mathscr{P}(V))\rightarrow\mathscr{P}(V))$
is a $\mathscr{P}$ interconnected $H^\infty$-algebra.
\begin{proposition}
    The $\mathscr{P}$ interconnected $H^\infty$-algebra $(\mathscr{P}(V),\gamma(V))$ equipped with $\eta(V):V\rightarrow\mathscr{P}(V)$
    is the free $\mathscr{P}$ interconnected $H^\infty$-algebra over $V$.
\end{proposition}
\begin{proof}
    For any $H^\infty$-map $f:V\rightarrow A$, where $A$ is a $\mathscr{P}$ interconnected $H^\infty$-algebra, we consider the composition
    $\widetilde{f}:\mathscr{P}(V)\xrightarrow{\mathscr{P}(f)}\mathscr{P}(A)\xrightarrow{\gamma_A} A$.
    It extends $f$ since the composite
    \[
        V\xrightarrow{\eta(V)}\mathscr{P}(V)\xrightarrow{\mathscr{P}(f)}\mathscr{P}(A)\xrightarrow{\gamma_A} A
    \]
    is $f$ by $\mathscr{P}(f)\circ\eta(V)=\eta(A)\circ f$ and $\gamma_{A}\circ\eta(A)=I_A$.

    The following diagram is commutative by functoriality and the fact that $A$ is a $\mathscr{P}$ interconnected $H^\infty$-algebra:
    \begin{center}
        \begin{tikzpicture}
            \node(A) at (0,0) {$\mathscr{P}(\mathscr{P}(V))$};
            \node(B) at (3,0) {$\mathscr{P}(V)$};
            \node(C) at (0,-1.5) {$\mathscr{P}(\mathscr{P}(A))$};
            \node(D) at (3,-1.5) {$\mathscr{P}(A)$};
            \node(E) at (0,-3) {$\mathscr{P}(A)$};
            \node(F) at (3,-3) {$A$};
            \draw[->] (A) -- (B);
            \draw[->] (A) -- (C);
            \draw[->] (B) -- (D);
            \draw[->] (C) -- (D);
            \draw[->] (C) -- (E);
            \draw[->] (E) -- (F);
            \draw[->] (D) -- (F);
            \draw[->][bend left=60] (B.east) to node[right]{$\widetilde{f}$} (F.east);
        \end{tikzpicture}
    \end{center}
    It implies that the $H^\infty$-map $\widetilde{f}$ is a morphism of $\mathscr{P}$ interconnected $H^\infty$-algebras.

    Let us show that the $H^\infty$-map $\widetilde{f}$ is unique. Since we want $\widetilde{f}$ to coincide with $f$ on $V$ and we want $\widetilde{f}$ to be
    a morphism of $\mathscr{P}$ interconnected $H^\infty$-algebras, there is no other choice by $\widetilde{f}$.
\end{proof}
\begin{remark}
    The possible free obejct over a left $H$-module $V$ in the category of $\mathscr{P}\ H$-pseudoalgebras
    must be a subobject of the free object $\mathscr{P}(V^\infty)$ over the regular $H^\infty$-module $V^\infty$
    in the category of $\mathscr{P}$ interconnected $H^\infty$-algebras.
\end{remark}


\begin{thebibliography}{99}

    \footnotesize\itemsep=0pt

    \bibitem{BDK}
    \textsc{B. Bakalov, A. D'Andrea, V. Kac}:
    Theory of finite pseudoalgebras.
    \emph{Adv. Math.}
    Vol.162 (2001) no.1 1--140.

    \bibitem{BDK1}
    \textsc{B. Bakalov, A. D' Andrea, V. Kac}:
    Irreducible modules over finite simple Lie pseudoalgebras I. Primitive pseudoalgebras of type W and S.
    \emph{Adv. Math.}
    Vol.204 (2006) 278--346.

    \bibitem{BDK2}
    \textsc{B. Bakalov, A. D' Andrea, V. Kac}:
    Irreducible modules over finite simple Lie pseudoalgebras II. Primitive pseudoalgebras of type K.
    \emph{Adv. Math.}
    Vol.232 (2013) 188--237.

    \bibitem{BDK3}
    \textsc{B. Bakalov, A. D' Andrea, V. Kac}:
    Irreducible modules over finite simple Lie pseudoalgebras III. Primitive pseudoalgebras of type H.
    \emph{Adv. Math.}
    Vol.392 (2021) 0001--8708.

    \bibitem{Ba}
    \textsc{B. Bakalov et al.}:
    An operadic approach to vertex algebra and Poisson vertex algebra cohomology.
    \emph{Jpn. J. Math.}
    Vol.14 (2019), no.2, 249--342.

    \bibitem{Ba2}
    \textsc{B. Bakalov et al.}:
    Chiral versus classical operad.
    \emph{Int. Math. Res. Not.}
    Vol.2020, no.19, 6463--6488.

    \bibitem{BD}
    \textsc{A. Beilinson, V. Drinfeld}:
    Chiral algebras.
    American Mathematical Society, 2004.

    \bibitem{Borcherds}
    \textsc{R. Borcherds}:
    Vertex algebras, Kac-Moody algebras, and the Monster.
    \emph{Proc. Nat. Acad. Sci. U.S.A.}
    Vol.83 (1986) no.10 3068-3071.
    
    \bibitem{B2}
    \textsc{R. Borcherds}:
    Vertex algebras, Topological field theory, primitive forms and related topics(M. Kashiwara, A. Matsuo, K. Saito, and I. Satake, eds.).
    Birkh\"auser, Boston, 1998, pp. 35-77.

    \bibitem{BSW}
    \textsc{M. Benini, A. Schenkel and L. Woike}:
    Operads for algebraic quantum field theory.
    \emph{Commun. Contemp. Math.}
    Vol.23 (2021), no.2, No. 2050007, 39 pp.

    \bibitem{FG}
    \textsc{J.~N.~K. Francis and D. Gaitsgory}:
    Chiral Koszul duality.
    \emph{Selecta Math. (N.S.)}
    Vol.18 (2012), no.1, 27--87.

    \bibitem{Kac}
    \textsc{V. Kac}:
    Vertex algebras for beginners.
    Univ. Lecture Ser., 10
    American Mathematical Society, Providence, RI, 1997, viii+141 pp.
    ISBN: 0-8218-0643-2

    \bibitem{Kolesnikov}
    \textsc{P. Kolesnikov}:
    Simple finite Jordan pseudoalgebras.
    \emph{SIGMA Symmetry Integrability Geom. Methods Appl.}
    Vol.5 (2009) Paper 014, 17.

    \bibitem{La}
    \textsc{J. Lambek}:
    Deductive systems and categories.
    Lecture Notes in Math., vol.86
    Springer-Verlag, Berlin-New York-Heidelberg, 1969, pp. 76-122.

    \bibitem{Li}
    \textsc{F. E. J. Linton}:
    The multilinear Yoneda lemmas: Toccata, fugue, and fantasia on themes by Eilenberg-Kelly and Yoneda.
    Lecture Notes in Math., vol.195
    Springer-Verlag, Berlin-New York-Heidelberg, 1971, pp. 209-229.

    \bibitem{LV}
    \textsc{J.L. Loday, B. Vallette}:
    Algebraic operads.
    Grundlehren Math. Wiss., 346[Fundamental Principles of Mathematical Sciences]
    Springer, Heidelberg, 2012, xxiv+634 pp.
    ISBN: 978-3-642-30361-6
    
    \bibitem{M}
    \textsc{J.P. May}:
    The geometry of iterated loop spaces.
    Springer-Verlag, Berlin, 1972, Lectures Notes in Mathematics, Vol.271.

    \bibitem{NY}
    \textsc{Y. Nishinaka and S. Yanagida}:
    Algebraic operads of SUSY vertex algebras and SUSY Poisson vertex algebras.
    \emph{Adv. Math.}
    Vol.483 (2025) No.110671, 115 pp.

    \bibitem{RA}
    \textsc{A.Retakh}:
    Unital associative pseudoalgebras and their representations.
    \emph{J. Algebra}
    Vol.277 {2004} no.2 769--805.

    \bibitem{Wu1}
    \textsc{Wu Zhixiang}:
    Leibniz {$H$}-pseudoalgebras.
    \emph{J. Algebra}
    Vol.437 (2015) 1--33.

    \bibitem{Wu2}
    \textsc{Wu Zhixiang}:
    Graded left symmetric pseudo-algebras.
    \emph{Comm. Algebra}
    Vol.43 (2015) no.9 3869--3897.

    \bibitem{Y}
    \textsc{D. Yau}:
    Colored operads.
    Graduate Studies in Mathematics, 170,
    Amer. Math. Soc., Providence, RI, 2016

    \bibitem{YR}
    \textsc{Yao Rui, Wu Zhixiang}:
    Some results in Zinbiel {$H$}-pseudoalgebras.
    \emph{J. Algebra Appl.}
    Vol.24 (2025) no.11 2550256.


\end{thebibliography}
\end{document}